\documentclass[12pt]{article}
\usepackage[english]{babel}
\usepackage{amsmath,amsthm}
\usepackage{amsfonts}
\usepackage{enumerate}
\usepackage{hyperref}
\hypersetup{colorlinks,
            linkcolor=blue, anchorcolor=blue,
            citecolor=blue, urlcolor=blue}
\newtheorem{thm}[equation]{Theorem}

\newtheorem{lem}[equation]{Lemma}

\newtheorem{claim}[equation]{Claim}

\theoremstyle{definition}
\newtheorem{defn}[equation]{Definition}
\newtheorem{example}[equation]{Example}
\newtheorem{observation}[equation]{Observation}
\newtheorem*{notation}{Notation}

\numberwithin{equation}{section}
\begin{document}
\title{\bf\Large A note on Fujimoto's uniqueness theorem with $(2n+3)$ hyperplanes}%
\author{Kai Zhou \footnote{{\it E-mail address}: \texttt{zhoukai@tongji.edu.cn}}}%
\date{}%
\maketitle
\def\thefootnote{}
% ----------------------------------------------------------------
\begin{abstract}
  Hirotaka Fujimoto proved in [Nagoya Math. J., 1976(64): 117--147] and [Nagoya Math. J., 1978(71): 13--24] a uniqueness theorem for algebraically non-degenerate meromorphic maps into $\mathbb{P}^n(\mathbb{C})$ sharing $(2n+3)$ hyperplanes in general position. The proof of Lemma 3.6 in [Nagoya Math. J., 1976(64): 117--147] is found to contain a mistake. This mistake can be corrected easily from a new point of view. This note explains the idea of the correction and provides a complete proof. 
\end{abstract}
% ----------------------------------------------------------------
\section{Introduction}          \label{sec:Introduction}

In the 1970s, Hirotaka Fujimoto \cite{Fujimoto75,Fujimoto76,Fujimoto78} proved several uniqueness theorems for meromorphic maps into a complex projective space, which are generalizations of the uniqueness theorems of meromorphic functions obtained in \cite{Polya1921} and \cite{Nevanlinna1926}. For instance, Fujimoto proved in \cite{Fujimoto76} and \cite{Fujimoto78} the following uniqueness theorem.
\begin{thm}        \label{thm:the_uniqueness_thm_with_(2n+3)_hyperplanes}
Let $ f $ and $ g $ be two meromorphic maps of $\mathbb{C}^m $ into the $ n$-dimensional complex projective space $\mathbb{P}^n(\mathbb{C})$ such that $ f(\mathbb{C}^m)\not\subseteq H_j $, $ g(\mathbb{C}^m)\not\subseteq H_j $ and $ f^*(H_j)=g^*(H_j)\, (1\leq j\leq 2n+3)$ for $(2n+3)$ hyperplanes $ H_j $ in $\mathbb{P}^n(\mathbb{C})$ in general position. If $ g $ is algebraically non-degenerate, then $ f=g.$
\end{thm}

One of the key ingredients of Fujimoto's proof of the above theorem is the Lemma 3.6 in \cite{Fujimoto76}. Unfortunately, a mistake is found in Fujimoto's proof of the Lemma 3.6 in \cite{Fujimoto76}. In this note, I shall show this mistake and explain the idea of correcting it. A complete proof of the Lemma 3.6 in \cite{Fujimoto76} will be given in Section \ref{sec:Proof_of_3rdCombiLem}.

To state Fujimoto's Lemma 3.6 in \cite{Fujimoto76}, we need some preparations.

Let $ G $ be a torsion-free abelian group and let $ A=(\alpha_1,\dots,\alpha_q)$ be a $ q$-tuple of elements in $ G.$ Denote by $\langle\alpha_1,\dots, \alpha_q \rangle $ the subgroup of $ G $ generated by $\alpha_1,\dots,\alpha_q.$ Since $ G $ is torsion-free, the subgroup $\langle\alpha_1,\dots, \alpha_q \rangle $ is free and is of finite rank. We denote by {\rm rank}$\{\alpha_1,\dots,\alpha_q\}$ the rank of $\langle\alpha_1,\dots, \alpha_q \rangle.$

The following notion is introduced by Fujimoto.
\begin{defn}      \label{defn:property_(Pr,s)}
Let $ q\geq r>s\geq 1 $ be positive integers. Let $(G,\cdot)$ be a torsion-free abelian group. A $ q$-tuple $ A=(\alpha_1,\dots,\alpha_q)$ of elements in $ G $ is said to \emph{have the property} $(P_{r,s})$ if arbitrarily chosen $ r $ elements $\alpha_{l(1)},\dots,\alpha_{l(r)}$ in $ A $ ($ 1\leq l(1)<\dots<l(r)\leq q $) satisfy the condition that, for any $ s $ distinct indices $ i_1,\dots, i_s \in\{1,\dots,r\},$ there exist distinct indices $ j_1,\dots, j_s \in\{1,\dots,r\}$ with $\{j_1,\dots,j_s\}\neq \{i_1,\dots,i_s\}$ such that
 \[
    \alpha_{l(i_1)}\cdot\alpha_{l(i_2)}\cdots \alpha_{l(i_s)}= \alpha_{l(j_1)}\cdot\alpha_{l(j_2)}\cdots \alpha_{l(j_s)}.
 \]
\end{defn}

We also need the following notation.
\begin{notation}
For elements $\alpha_1,\alpha_2,\dots,\alpha_q,\tilde{\alpha}_1,\tilde{\alpha}_2,\dots,\tilde{\alpha}_q $ in an abelian group $(G,\cdot),$ we write
 \[
    \alpha_1:\alpha_2:\dots:\alpha_q=\tilde{\alpha}_1:\tilde{\alpha}_2:\dots:\tilde{\alpha}_q,
 \]
if there is an element $\beta\in G $ such that $\alpha_i=\beta\tilde{\alpha}_i $ for any $ 1\leq i\leq q.$
\end{notation}

Now the Lemma 3.6 in \cite{Fujimoto76} is stated as follows.
\begin{lem}      \label{lem:the_third_combinatorial_lemma}
Let $ s $ and $ q $ be positive integers with $ 2\leq s<q\leq 2s.$ Let $(G,\cdot)$ be a torsion-free abelian group. Let $ A=(\alpha_1,\dots,\alpha_q)$ be a $ q$-tuple of elements in $ G $ that has the property $(P_{q,s}),$ and assume that at least one $\alpha_i $ equals the unit element $ 1 $ of $ G.$ Then
 \begin{enumerate}[\rm (i)]
   \item {\rm rank}$\{\alpha_1,\dots,\alpha_q\}=:t\leq s-1;$
   \item if $ t=s-1,$ then $ q=2s $ and there is a basis $\{\beta_1,\dots,\beta_{s-1}\}$ of $\langle\alpha_1,\dots,\alpha_q\rangle $ such that the $\alpha_i $ are represented, after a suitable change of indices, as one of the following two types:
     \begin{itemize}
       \item[\rm (A)] $ s $ is odd and
            \[
               \alpha_1:\alpha_2:\dots:\alpha_{2s}= 1:1:\beta_1:\beta_1:\beta_2:\beta_2:\dots:\beta_{s-1}:\beta_{s-1};
            \]
       \item[\rm (B)] $\alpha_1:\alpha_2:\dots:\alpha_{2s}= 1:1:\dots:1:\beta_1:\dots:\beta_{s-1}: (\beta_1\cdots\beta_{a_1})^{-1}:(\beta_{a_1+1}\cdots\beta_{a_2})^{-1}:\dots:(\beta_{a_{k-1}+1}\cdots\beta_{a_k})^{-1},$
           \par where $ 0\leq k\leq s-1 $, $ 1\leq a_1<a_2<\dots<a_k\leq s-1,$ and the unit element $ 1 $ appears $(s+1-k)$ times in the right hand side.
     \end{itemize}
 \end{enumerate}
\end{lem}

Fujimoto's proof of the above lemma essentially uses the assertion that, for any $ q$-tuple $ A=(\alpha_1,\dots,\alpha_q)$ of elements in a torsion-free abelian group $(G,+),$ there exists a basis $\{\eta_1,\dots,\eta_t\}$ of $\langle\alpha_1,\dots,\alpha_q\rangle $ with the following property: there are distinct indices $ i_1,\dots, i_t\in\{1,\dots,q\}$ and nonzero integers $ l_1,\dots, l_t $ such that
\[
   \alpha_{i_{\tau}}=l_{\tau}\eta_{\tau},\quad 1\leq \tau\leq t.
\]
Such a basis $\{\eta_1,\dots,\eta_t\}$ is called by Fujimoto an adequate basis for $ A.$ Unfortunately, adequate basis does not exist in general. This is illustrated by the following example.

\begin{example}            \label{example:no_adequate_basis}
Consider the free abelian group $(\mathbb{Z}^3,+).$ Let
 \[
    \alpha_1=(1,0,0),\quad \alpha_2=(1,1,0),\quad \alpha_3=(1,2,2),\quad \alpha_4=(1,2,5)
 \]
be four elements in $\mathbb{Z}^3.$

Since
 \[
    (0,1,0)=\alpha_2-\alpha_1\quad\text{and}\quad (0,0,1)=-2\alpha_3+\alpha_4+2\alpha_2-\alpha_1,
 \]
we see that
 \[
    \langle\alpha_1,\alpha_2,\alpha_3,\alpha_4\rangle=\mathbb{Z}^3.
 \]
So, if there exists an adequate basis for $(\alpha_1,\alpha_2,\alpha_3,\alpha_4),$ then one sees easily that there are distinct indices $ i,j,k\in\{1,2,3,4\}$ such that $\{\alpha_i,\alpha_j,\alpha_k\}$ is a basis of $\mathbb{Z}^3.$ However, one verifies easily that $\{\alpha_i,\alpha_j,\alpha_k\}$ is not a basis of $\mathbb{Z}^3 $ for any three distinct indices $ i,j,k\in\{1,2,3,4\}.$ This shows that adequate basis does not exist for $(\alpha_1,\alpha_2,\alpha_3,\alpha_4).$
\end{example}

Fujimoto's proof of Lemma \ref{lem:the_third_combinatorial_lemma} can be corrected easily from a new point of view. Roughly speaking, we take an ``adequate basis'' in $\langle\alpha_1,\dots,\alpha_q\rangle\otimes_{\mathbb{Z}}\mathbb{Q}$ and regard $\langle\alpha_1,\dots,\alpha_q\rangle\otimes_{\mathbb{Z}}\mathbb{Q}$ as the ambient space. I shall explain more.

Let $(G,+)$ be a torsion-free abelian group and let $ A=(\alpha_1,\dots,\alpha_q)$ be a $ q$-tuple of elements in $ G.$ We regard the free abelian group $\langle\alpha_1,\dots,\alpha_q\rangle $ as a subgroup of $\langle\alpha_1,\dots,\alpha_q\rangle\otimes_{\mathbb{Z}}\mathbb{Q}$ which can be naturally regarded as a $\mathbb{Q}$-vector space. Because the $\mathbb{Q}$-vector space $\langle\alpha_1,\dots,\alpha_q\rangle\otimes_{\mathbb{Z}}\mathbb{Q}$ is generated by $\alpha_1,\dots,\alpha_q,$ there exist distinct indices $ i_1,\dots,i_t\in\{1,\dots,q\}$ with $ t=\mbox{\rm rank}\{\alpha_1,\dots,\alpha_q\}$ such that $\{\alpha_{i_1},\dots,\alpha_{i_t}\}$ is a basis of the $\mathbb{Q}$-vector space $\langle\alpha_1,\dots,\alpha_q\rangle \otimes_{\mathbb{Z}}\mathbb{Q}.$ Then we see that there exist positive integers $ l_1,\dots,l_t $ such that, for every $ j\in\{1,\dots,q\}\setminus\{i_1,\dots,i_t\},$
\[
   \alpha_j= l(j,1)\cdot\frac{1}{l_1}\alpha_{i_1}+\dots+l(j,t)\cdot\frac{1}{l_t}\alpha_{i_t}
\]
for suitable integers $ l(j,1),\dots,l(j,t).$

The above discussion shows the following.
\begin{observation}      \label{obsv:adequate_basis_exists_by_tensor_Q}
Let $(G,+)$ be a torsion-free abelian group and let $ A=(\alpha_1,\dots,\alpha_q)$ be a $ q$-tuple of elements in $ G.$ We regard the free abelian group $\langle\alpha_1,\dots,\alpha_q\rangle $ as a subgroup of $\langle\alpha_1,\dots,\alpha_q\rangle\otimes_{\mathbb{Z}}\mathbb{Q}.$ There exist distinct indices $ i_1,\dots,i_t\in\{1,\dots,q\}$ with $ t=\mbox{\rm rank}\{\alpha_1,\dots,\alpha_q\},$ positive integers $ l_1,\dots,l_t $ and a basis $\{\eta_1,\dots,\eta_t\}$ of the $\mathbb{Q}$-vector space $\langle\alpha_1,\dots,\alpha_q\rangle\otimes_{\mathbb{Z}}\mathbb{Q},$ such that
 \begin{align*}
    \alpha_{i_{\tau}}&=l_{\tau}\eta_{\tau}, \quad 1\leq \tau\leq t,
 \\ \alpha_j &=l(j,1)\eta_1+\dots+l(j,t)\eta_t, \quad j\in\{1,\dots,q\}\setminus\{i_1,\dots,i_t\},
 \end{align*}
where $ l(j,1),\dots,l(j,t)$ are integers.
\end{observation}

A complete proof of Lemma \ref{lem:the_third_combinatorial_lemma} will be given in Section \ref{sec:Proof_of_3rdCombiLem}. In the proof, we take a basis $\{\eta_1,\dots,\eta_t\}$ as in Observation \ref{obsv:adequate_basis_exists_by_tensor_Q}, and we regard the $\alpha_i $ as elements in $\langle\alpha_1,\dots,\alpha_q\rangle\otimes_{\mathbb{Z}}\mathbb{Q}$ whenever it is needed. The proof basically follows Fujimoto's original strategy. But simplifications are made in some steps of Fujimoto's strategy.

We shall give another two observations for later use.

\begin{observation}      \label{obsv:two_tuples_containing_1_proportion_generate_same group}
Let $(\alpha_1,\dots,\alpha_q)$ and $(\tilde{\alpha}_1,\dots,\tilde{\alpha}_q)$ be two $ q$-tuples of elements in a torsion-free abelian group $(G,\cdot).$ Assume that there exist indices $ i_0 $ and $ j_0 $ such that $\alpha_{i_0}=\tilde{\alpha}_{j_0}=1,$ and there is an element $\beta\in G $ such that $\alpha_i= \beta\tilde{\alpha}_i $ for any $ 1\leq i\leq q.$ Then
 \[
    \langle\alpha_1,\dots,\alpha_q\rangle= \langle\tilde{\alpha}_1,\dots,\tilde{\alpha}_q\rangle.
 \]
\end{observation}

\begin{observation}    \label{obsv:type(B)_is_symmetry_for_nonunit_elements}
Let $ x_1,\dots,x_{t+k}$ be elements in a torsion-free abelian group $(G,\cdot),$ which can be represented as follows:
 \begin{align*}
    &(x_1,\dots,x_t,x_{t+1},\dots,x_{t+k})
 \\ &\qquad =\big(\beta_1,\dots,\beta_t, (\beta_1\cdots\beta_{a_1})^{-1}, (\beta_{a_1+1}\cdots\beta_{a_2})^{-1},\dots,(\beta_{a_{k-1}+1}\cdots\beta_{a_k})^{-1}\big),
 \end{align*}
where $ t\geq 1 $, $ 0\leq k\leq t $, $ 1\leq a_1<\dots<a_k\leq t,$ and $\beta_1,\dots,\beta_t $ are multiplicatively independent elements in $ G.$ For any $ t $ multiplicatively independent elements $ x_{i_1},\dots,x_{i_t}$ among $ x_i $'s, by writing $\{1,\dots,t+k\}\setminus\{i_1,\dots,i_t\}=\{j_1,\dots,j_k\},$ we can verify that
 \[
    (x_{i_1},\dots,x_{i_t},x_{j_1},\dots,x_{j_k})= \big(x_{i_1},\dots,x_{i_t}, \big(\prod_{i\in\Lambda_1}x_i\big)^{-1},\dots,\big(\prod_{i\in\Lambda_k}x_i\big)^{-1}\big),
 \]
where $\Lambda_1,\dots,\Lambda_k $ are nonempty subsets of $\{i_1,\dots,i_t\}$ with $\Lambda_i\cap\Lambda_j=\emptyset $ for any $ 1\leq i<j\leq k.$
\end{observation}

We also recall the following result of Fujimoto.
\begin{lem}[see Lemma 2.6 in \cite{Fujimoto75}]        \label{lem:1st_CombiLem_simple}
Let $ G $ be a torsion-free abelian group. Let $ A=(\alpha_1,\dots,\alpha_q)$ be a $ q$-tuple of elements in $ G $ that has the property $(P_{r,s}),$ where $ q\geq r>s\geq 1.$ Then there exist two distinct indices $ i,j\in\{1,\dots,q\}$ such that $\alpha_i=\alpha_j.$
\end{lem}

The following notation will be used in Section \ref{sec:Proof_of_3rdCombiLem}.
\begin{notation}
For a set $ S,$ we denote by $\#S $ the cardinality of $ S.$
\end{notation}

\section{Proof of Lemma \ref{lem:the_third_combinatorial_lemma}}       \label{sec:Proof_of_3rdCombiLem}

Lemma \ref{lem:the_third_combinatorial_lemma} will be proved by the induction on $ s.$

Consider first the case of $ s=2.$ In this case, $ q=3 $ or $ q=4.$

When $ q=3,$ we claim that $\alpha_1=\alpha_2=\alpha_3.$ Assume this is not true. Without loss of generality, we may assume that $\alpha_1\neq\alpha_2.$ Because $\alpha_1\alpha_3\neq\alpha_2\alpha_3,$ by assumption, we see $\alpha_1\alpha_3=\alpha_1\alpha_2,$ which implies $\alpha_3=\alpha_2.$
Similarly, we get $\alpha_3=\alpha_1.$ Then $\alpha_1=\alpha_2,$ which gives a contradiction. Thus $\alpha_1=\alpha_2=\alpha_3=1 $ and
{\rm rank}$\{\alpha_1,\alpha_2,\alpha_3\}=0.$ This shows the conclusions of Lemma \ref{lem:the_third_combinatorial_lemma} hold in this case.

Now assume $ q=4.$ It follows from Lemma \ref{lem:1st_CombiLem_simple} that there are distinct indices $ i_0,j_0\in\{1,\dots,4\}$ such that $\alpha_{i_0}=\alpha_{j_0}.$
Then by Observation \ref{obsv:two_tuples_containing_1_proportion_generate_same group} and considering the new 4-tuple $(\alpha_{i_0}^{-1}\alpha_1,\dots,\alpha_{i_0}^{-1}\alpha_4),$ we may assume, after a suitable change of indices, that $\alpha_1=\alpha_2=1.$
By assumption, there are distinct indices $ i,j\in\{1,\dots,4\}$ with $\{i,j\}\neq\{1,2\}$ such that
$\alpha_1\alpha_2= \alpha_i\alpha_j.$ If $\{i,j\}\neq\{3,4\},$ then we see easily that $\alpha_3=1 $ or $\alpha_4=1.$ If $\{i,j\}=\{3,4\},$ then $\alpha_3=\alpha_4^{-1}.$ So at least one of the following three cases occurs:
\begin{itemize}
  \item $(\alpha_1,\alpha_2,\alpha_3,\alpha_4)=(1,1,1,\beta);$
  \item $(\alpha_1,\alpha_2,\alpha_3,\alpha_4)=(1,1,\beta,1);$
  \item $(\alpha_1,\alpha_2,\alpha_3,\alpha_4)=(1,1,\beta,\beta^{-1}).$
\end{itemize}
It follows that {\rm rank}$\{\alpha_1,\dots,\alpha_4\}\leq 1,$ and if {\rm rank}$\{\alpha_1,\dots,\alpha_4\}=1,$ then $(\alpha_1,\dots,\alpha_4)$ is of the type (B).

Now Lemma \ref{lem:the_third_combinatorial_lemma} is proved in the case of $ s=2.$

Next we assume that $ s\geq 3 $ and assume the conclusions hold when $ s $ in Lemma \ref{lem:the_third_combinatorial_lemma} is replaced by a number $\leq s-1.$

If {\rm rank}$\{\alpha_1,\dots,\alpha_q\}=:t<s-1,$ then the conclusions of Lemma \ref{lem:the_third_combinatorial_lemma} hold trivially. Therefore, we assume in the following that $ t\geq s-1.$

Set
\[
   M_0:=\big\{i\in\{1,\dots,q\}\,|\,\alpha_i=1\big\} \quad\mbox{and}\quad m_0:=\# M_0.
\]
By Lemma \ref{lem:1st_CombiLem_simple} and Observation \ref{obsv:two_tuples_containing_1_proportion_generate_same group}, we may assume that $ m_0\geq 2.$

We may regard $\langle\alpha_1,\dots,\alpha_q \rangle $ as a subgroup of $\langle\alpha_1,\dots,\alpha_q \rangle\otimes_{\mathbb{Z}}\mathbb{Q}.$ By Observation \ref{obsv:adequate_basis_exists_by_tensor_Q}, there exist distinct indices $ u_1,\dots,u_t\in\{1,\dots,q\},$ positive integers $ l_1,\dots,l_t,$ and a basis $\{\eta_1,\dots,\eta_t\}$ of the $\mathbb{Q}$-vector space $\langle\alpha_1,\dots,\alpha_q \rangle\otimes_{\mathbb{Z}}\mathbb{Q},$ such that, the $\alpha_i,$ as elements in $\langle\alpha_1,\dots,\alpha_q \rangle \otimes_{\mathbb{Z}}\mathbb{Q},$ are represented as follows:
\begin{equation}      \label{equ:represent_alphai_by_an_adequate_basis_after_tensor_Q}
  \alpha_i=\eta_1^{l(i,1)}\cdot\cdots\cdot\eta_t^{l(i,t)}, \quad 1\leq i\leq q,
\end{equation}
where $ l(i,1),\dots,l(i,t)$ are integers and, for each $ 1\leq \tau\leq t,$
\[
   l(u_{\tau},\tau)=l_{\tau}(>0) \quad\mbox{and}\quad l(u_{\tau},\tau')=0 \,\,\mbox{for any}\,\, \tau'\neq \tau.
\]

The rest of the proof is divided into the following two cases:
\begin{description}
  \item[Case $(\alpha)$] For every $ 1\leq\tau\leq t,$ the $ q $ integers $ l(1,\tau),\dots,l(q,\tau)$ are all non-negative;
  \item[Case $(\beta)$] There is some $\tau\in\{1,\dots,t\}$ such that $ l(i,\tau)<0 $ for some $ i\in\{1,\dots,q\}.$
\end{description}

\subsection{The proof for the Case $(\alpha)$}        \label{ssec:Proof_for_the_Case(alpha)}

For each $ 1\leq \tau\leq t,$ we put
\[
   M_{\tau}:=\big\{i\in\{1,\dots,q\}\,|\,l(i,\tau)>0,\, l(i,\tau+1)=\dots=l(i,t)=0\big\}
\]
and
\[
   m_{\tau}:=\# M_{\tau}.
\]
Every $ M_{\tau}$ is nonempty, since $ u_{\tau}\in M_{\tau}.$

We note here that, for any $ i\in M_0 $, $ l(i,1)=l(i,2)=\dots=l(i,t)=0.$ So these $ M_{\tau}$ give the following partition of $\{1,\dots,q\}$:
\[
   \{1,\dots,q\}=M_0\cup M_1\cup \cdots\cup M_t.
\]

\begin{claim}      \label{claim:the_sum_of_any_mx_is_neq_s}
For any subset $\{\tau_1,\dots,\tau_p\}$ of $\{0,1,\dots, t\}$ with $ \tau_1<\tau_2<\dots<\tau_p,$
 \[
    m_{\tau_1}+\dots+ m_{\tau_p}\neq s.
 \]
\end{claim}
\begin{proof}
Assume on the contrary that $ m_{\tau_1}+\dots+ m_{\tau_p}=s $ for some $\tau_1,\dots,\tau_p $ with $ 0\leq\tau_1<\dots<\tau_p\leq t.$ Write
\[
   M_{\tau_1}\cup\cdots\cup M_{\tau_p}=\{i_1,\dots,i_s\}=:I.
\]

Consider the $ s $ elements $\alpha_{i_1},\dots,\alpha_{i_s}.$ By assumption, there exist distinct indices $ j_1,\dots,j_s\in\{1,\dots,q\}$ with $\{j_1,\dots,j_s\}=:J\neq I,$ such that
\begin{equation}     \label{equ:alphaix=alphajx_in_proof_of_any_mx_is_neq_s}
  \alpha_{i_1}\cdot\dots\cdot\alpha_{i_s}= \alpha_{j_1}\cdot\dots\cdot\alpha_{j_s}.
\end{equation}
We regard the $\alpha_i $ as elements in $\langle\alpha_1,\dots,\alpha_q \rangle\otimes_{\mathbb{Z}}\mathbb{Q}$ and substitute \eqref{equ:represent_alphai_by_an_adequate_basis_after_tensor_Q} into the above equation. Then we get
\begin{equation}     \label{equ:alphaix=alphajx_represent_by_using_adequate_basis}
  \eta_1^{l(i_1,1)+\dots+l(i_s,1)}\cdots\eta_t^{l(i_1,t)+\dots+l(i_s,t)}= \eta_1^{l(j_1,1)+\dots+l(j_s,1)}\cdots\eta_t^{l(j_1,t)+\dots+l(j_s,t)}.
\end{equation}
For each $ 1\leq\tau\leq t,$ the exponent of $\eta_{\tau}$ in the left hand side of \eqref{equ:alphaix=alphajx_represent_by_using_adequate_basis} is necessarily equal to that in the right hand side.

We consider the following three cases.

$\bullet\quad \tau_1\geq 1.$

Let $\tau\in\{\tau_p+1,\dots,t\}.$ Since $ I=M_{\tau_1}\cup\cdots\cup M_{\tau_p},$ the exponent of $\eta_{\tau}$ in the left hand side of \eqref{equ:alphaix=alphajx_represent_by_using_adequate_basis} is zero, and thus $ J\cap M_{\tau}=\emptyset.$ For, if $ J\cap M_{\tau}\neq\emptyset,$ then the exponent of $\eta_{\tau}$ in the right hand side of \eqref{equ:alphaix=alphajx_represent_by_using_adequate_basis} is strictly greater than zero, and this gives a contradiction.

Then by observing the exponents of $\eta_{\tau_p}$ in the both sides of \eqref{equ:alphaix=alphajx_represent_by_using_adequate_basis}, we see $ J\supseteq M_{\tau_p}.$ For, if $ J\not\supseteq M_{\tau_p},$ then the exponent of $\eta_{\tau_p}$ in the right hand side is smaller than $\sum_{i\in M_{\tau_p}}l(i,\tau_p)$ which is the exponent of $\eta_{\tau_p}$ in the left hand side, and this is a contradiction.

In summary, we have showed that
\[
   J\cap (M_{\tau_p+1}\cup\cdots\cup M_t)= \emptyset \quad\mbox{and}\quad J\supseteq M_{\tau_p}.
\]

Cancelling the elements $\alpha_i\,(i\in M_{\tau_p})$ in the both sides of \eqref{equ:alphaix=alphajx_in_proof_of_any_mx_is_neq_s}, we get
\[
   \prod_{i\in M_{\tau_1}\cup\cdots\cup M_{\tau_{p-1}}} \alpha_i= \prod_{i\in J\setminus M_{\tau_p}} \alpha_i.
\]
As before, we substitute \eqref{equ:represent_alphai_by_an_adequate_basis_after_tensor_Q} into the above equation. By comparing the exponents of $\eta_{\tau}\,(\tau_{p-1}<\tau<\tau_p)$ in the both sides, we see that
\[
   J\cap \bigcup_{\tau_{p-1}<\tau<\tau_p} M_{\tau}=\emptyset.
\]
Then by comparing the exponents of $\eta_{\tau_{p-1}}$ in the both sides, we see
\[
   J\supseteq M_{\tau_{p-1}}.
\]
Then cancelling the elements $\alpha_i\,(i\in M_{\tau_p}\cup M_{\tau_{p-1}})$ in the both sides of \eqref{equ:alphaix=alphajx_in_proof_of_any_mx_is_neq_s}, we get
\[
   \prod_{i\in M_{\tau_1}\cup\cdots\cup M_{\tau_{p-2}}} \alpha_i= \prod_{i\in J\setminus (M_{\tau_p}\cup M_{\tau_{p-1}})} \alpha_i.
\]
Continuing this process, we finally get that
\[
   J\cap \bigcup_{\tau_i<\tau<\tau_{i+1}} M_{\tau}=\emptyset \quad\mbox{and}\quad J\supseteq M_{\tau_i}, \quad 1\leq i\leq p,
\]
where we put $\tau_{p+1}:=t+1.$ Since $\# J=\# I,$ we see
\[
   J=M_{\tau_1}\cup\cdots\cup M_{\tau_p}=I,
\]
which gives a contradiction.

$\bullet\quad \tau_1=0 $ and $ p\geq 2.$

The similar argument as before shows that
\[
   J\cap \bigcup_{\tau_i<\tau<\tau_{i+1}} M_{\tau}=\emptyset \quad\mbox{and}\quad J\supseteq M_{\tau_i}, \quad 2\leq i\leq p.
\]
Then cancelling the elements $\alpha_i\,(i\in M_{\tau_2}\cup\cdots\cup M_{\tau_p})$ in the both sides of \eqref{equ:alphaix=alphajx_in_proof_of_any_mx_is_neq_s}, we get
\[
   1=\prod_{i\in M_0} \alpha_i= \prod_{i\in J\setminus (M_{\tau_2}\cup\cdots\cup M_{\tau_p})} \alpha_i.
\]
Then we see easily that
\[
   J\cap \bigcup_{1\leq\tau<\tau_2} M_{\tau}=\emptyset.
\]
Because $\# J=\# I=m_0+m_{\tau_2}+\cdots+m_{\tau_p},$ we get
\[
   J=M_0\cup M_{\tau_2}\cup\cdots\cup M_{\tau_p}=I,
\]
which gives a contradiction.

$\bullet\quad \tau_1=0 $ and $ p=1.$

In this case, \eqref{equ:alphaix=alphajx_in_proof_of_any_mx_is_neq_s} becomes
\[
   1=\prod_{i\in J} \alpha_i.
\]
We see that $ J\subseteq M_0.$ Then by $\# J=\# I,$ we get $ J=M_0=I,$ which also gives a contradiction.

In any case, we get a contradiction. Therefore the Claim \ref{claim:the_sum_of_any_mx_is_neq_s} is proved.
\end{proof}

We shall prove next the following.
\begin{claim}       \label{claim:two_types_of_mx}
$ t=s-1 $, $ q=2s,$ and one of the following two cases occurs:
 \begin{enumerate}[\rm (a)]
   \item $ m_0=s+1 $, $ m_1=m_2=\dots=m_{s-1}=1;$
   \item $ s $ is odd and $ m_0=m_1=\dots=m_{s-1}=2.$
 \end{enumerate}
\end{claim}

\begin{proof}
Let $\sigma:\{1,\dots,t\}\to\{1,\dots,t\}$ be a bijection such that
\[
   m_{\sigma(1)}\geq m_{\sigma(2)}\geq \cdots\geq m_{\sigma(t)}.
\]
We also put $\sigma(0)=0.$

Since $ m_{\sigma(0)}+ m_{\sigma(1)}+\cdots+ m_{\sigma(t)}=q>s,$ there is a minimum natural number $ p $ such that
\[
   m_{\sigma(0)}+m_{\sigma(1)}+\cdots+ m_{\sigma(p)}>s.
\]

We consider the following two cases.

$\bullet\quad p=0.$

$ p=0 $ means that $ m_0>s.$ Then since $ t\geq s-1 $ and every $ m_{\tau}$ ($ 1\leq\tau\leq t $) is a positive integer, we conclude that
\begin{align*}
   2s\geq q&=m_0+m_1+\dots+m_t\geq (s+1)+t\geq 2s.
\end{align*}
It follows that $ t=s-1 $, $ q=2s $ and $ m_0=s+1 $, $ m_1=m_2=\dots=m_t=1.$ This shows the case (a) occurs.
Hence Claim \ref{claim:two_types_of_mx} holds when $ p=0.$

$\bullet\quad p\geq 1.$

By the definition of $ p $ and using Claim \ref{claim:the_sum_of_any_mx_is_neq_s}, we get
\[
   m_{\sigma(0)}+m_{\sigma(1)}+\cdots+ m_{\sigma(p)}\geq s+1
\]
and
\[
   m_{\sigma(0)}+\cdots+ m_{\sigma(p-1)}=:\tilde{m}\leq s-1.
\]
In particular, we see
\[
   m_{\sigma(p)}\geq s+1-\tilde{m}\geq 2.
\]

Let $ r:=\#\big\{\tau\in\{1,\dots,t\}\,|\, m_{\tau}=1\big\}.$ Then $ 0\leq r\leq t-p $ and
\[
   m_{\sigma(t-r)}\geq 2,\quad m_{\sigma(t-r+1)}=\dots=m_{\sigma(t)}=1.
\]
By Claim \ref{claim:the_sum_of_any_mx_is_neq_s}, $\tilde{m}+r\leq s-1.$ For, if $\tilde{m}+r\geq s,$ by taking $(s-\tilde{m})$ distinct indices $ i_1,\dots, i_{s-\tilde{m}}$ in $\{\sigma(t-r+1),\dots,\sigma(t)\},$ we have
\[
   \tilde{m}+m_{i_1}+\dots+m_{i_{s-\tilde{m}}}=s,
\]
which contradicts Claim \ref{claim:the_sum_of_any_mx_is_neq_s}. Thus
\[
   r\leq s-1-\tilde{m}\leq m_{\sigma(p)}-2.
\]
Then by $ t\geq s-1,$ we conclude that
\begin{align*}
   2s\geq q&=m_0+m_{\sigma(1)}+\cdots+ m_{\sigma(p)}+(m_{\sigma(p+1)}+\cdots+m_{\sigma(t-r)})+r
\\ &\geq m_0+ p\cdot m_{\sigma(p)}+2(t-r-p)+r
\\ &= m_0+ p\cdot m_{\sigma(p)}+2t-2p-r
\\ &\geq m_0-2+(m_{\sigma(p)}-2)(p-1)+2s
\\ &\geq 2s.
\end{align*}
It follows that $ t=s-1 $, $ q=2s $, $ m_0=2 $, $ m_{\sigma(1)}=\dots=m_{\sigma(p)}$, $ r=s-1-\tilde{m}=m_{\sigma(p)}-2,$ and $(m_{\sigma(p)}-2)(p-1)=0.$

If $ m_{\sigma(p)}=2,$ then $ r=0 $ and $ m_{\sigma(1)}=\dots=m_{\sigma(t)}=2.$
If $ m_{\sigma(p)}\neq 2,$ then $ p=1 $ and $ m_{\sigma(1)}\geq 3.$ Thus $ r=m_{\sigma(1)}-2\geq 1 $ and $ m_{\sigma(1)}=s+1-\tilde{m}=s+1-m_0=s-1.$ It follows that
\[
   m_{\sigma(1)}+m_{\sigma(t)}=s,
\]
which contradicts Claim \ref{claim:the_sum_of_any_mx_is_neq_s}.
So it is impossible that  $ m_{\sigma(p)}\neq 2,$ and thus we have $ m_0=m_1=\dots=m_t=2.$

If $ s $ is even, then
\[
   m_1+m_2+\dots+m_{s/2}=s,
\]
which contradicts Claim \ref{claim:the_sum_of_any_mx_is_neq_s}. Thus $ s $ is odd and the case (b) occurs.
The proof of Claim \ref{claim:two_types_of_mx} is completed.
\end{proof}

By Claim \ref{claim:two_types_of_mx}, we know that $ t=s-1 $ and $ q=2s.$ So, the conclusion (i) of Lemma \ref{lem:the_third_combinatorial_lemma} is true, and for proving the conclusion (ii) of Lemma \ref{lem:the_third_combinatorial_lemma}, we need to show that $(\alpha_1,\dots,\alpha_{2s})$ is of the type (A) or type (B).

If the case (a) in Claim \ref{claim:two_types_of_mx} occurs, then
\[
   M_1=\{u_1\},\, M_2=\{u_2\},\dots, M_{s-1}=\{u_{s-1}\} \quad\mbox{and}\quad \# M_0=s+1.
\]
We see that $\alpha_i=1 $ for any $ i\in\{1,\dots,2s\}\setminus\{u_1,\dots,u_{s-1}\}.$
Put
\[
   \beta_1:=\alpha_{u_1},\,\beta_2:=\alpha_{u_2},\dots,\beta_{s-1}:=\alpha_{u_{s-1}}.
\]
Since
\[
   \alpha_{u_1}=\eta_1^{l_1},\,\alpha_{u_2}=\eta_2^{l_2},\dots,\alpha_{u_{s-1}}=\eta_{s-1}^{l_{s-1}},
\]
where $\{\eta_1,\dots,\eta_{s-1}\}$ is a basis of the $\mathbb{Q}$-vector space $\langle\alpha_1,\dots,\alpha_{2s}\rangle\otimes_{\mathbb{Z}}\mathbb{Q},$ we see that $\beta_1,\dots,\beta_{s-1}$ are multiplicatively independent elements in $ G,$ and thus $\{\beta_1,\dots,\beta_{s-1}\}$ is a basis of $\langle\alpha_1,\dots,\alpha_{2s}\rangle.$ After a suitable change of indices, we have
\[
   \alpha_1:\alpha_2:\dots:\alpha_{2s}= 1:1:\dots:1:\beta_1:\dots:\beta_{s-1},
\]
which is a special case of the type (B).

Now assume the case (b) in Claim \ref{claim:two_types_of_mx} occurs. We know that $ s $ is odd and $\# M_i=2 $ for any $ 0\leq i\leq s-1.$ We assume that
\[
   M_1=\{u_1,v_1\},\dots, M_{s-1}=\{u_{s-1},v_{s-1}\} \quad\mbox{and}\quad M_0=\{u_0,v_0\}.
\]
Regarding the $\alpha_i $ as elements in $\langle\alpha_1,\dots,\alpha_q \rangle\otimes_{\mathbb{Z}}\mathbb{Q},$ we have
\begin{equation}     \label{equ:represent_alphai_final_step_in_Case(alpha)}
  \alpha_{u_0}=\alpha_{v_0}=1 \quad\mbox{and}\quad \alpha_{u_{\tau}}=\eta_{\tau}^{l_{\tau}},\, \alpha_{v_{\tau}}=\eta_1^{l(v_{\tau},1)}\cdots\eta_{\tau}^{l(v_{\tau},{\tau})},\, 1\leq \tau\leq s-1,
\end{equation}
where $ l(v_{\tau},{\tau})>0 $ for any $ 1\leq \tau\leq s-1.$

We shall show that $\alpha_{u_{\tau}}=\alpha_{v_{\tau}}$ for any $ 1\leq \tau\leq s-1.$

Noting that $(s-1)/2 $ is a positive integer, we consider first the case when $ 1\leq \tau\leq (s-1)/2.$ Put
\[
   I:=\{u_{\tau}\}\cup M_{(s+1)/2}\cup\cdots\cup M_{s-1}.
\]
Then $\# I=s,$ and by assumption, there exists a $ J\subseteq\{1,\dots,2s\}$ with $\# J=s $ and $ J\neq I,$ such that
\begin{equation}     \label{equ:alphaI=alphaJ_in_proof_Case(alpha)_final}
   \prod_{j\in I} \alpha_j= \prod_{j\in J} \alpha_j.
\end{equation}
Using \eqref{equ:represent_alphai_final_step_in_Case(alpha)}, we represent the both sides of \eqref{equ:alphaI=alphaJ_in_proof_Case(alpha)_final} with $\eta_1,\dots,\eta_{s-1}.$ For $ i\in\{(s+1)/2,\dots,s-1\},$ by observing the exponents of $\eta_i $ in the both sides of \eqref{equ:alphaI=alphaJ_in_proof_Case(alpha)_final}, we see $ J\supseteq M_i.$ Then by $\# J=\# I,$ we know that $ J\cap (M_0\cup\cdots\cup M_{(s-1)/2})$ has exactly one element, say $ j_0.$ Cancelling the elements $\alpha_j\,(j\in M_{(s+1)/2}\cup\cdots\cup M_{s-1})$ in the both sides of \eqref{equ:alphaI=alphaJ_in_proof_Case(alpha)_final}, we get $\alpha_{u_{\tau}}=\alpha_{j_0}.$ Then by \eqref{equ:represent_alphai_final_step_in_Case(alpha)} and $ J\neq I,$ we see $ j_0=v_{\tau}.$ Therefore $\alpha_{u_{\tau}}=\alpha_{v_{\tau}}.$

Now let $(s+1)/2\leq \tau\leq s-1.$ Put
\[
   I:=M_0\cup\cdots\cup M_{(s-3)/2}\cup\{u_{\tau}\}.
\]
Then $\# I=s,$ and by assumption, there exists a $ J\subseteq\{1,\dots,2s\}$ with $\# J=s $ and $ J\neq I,$ such that
\begin{equation}     \label{equ:alphaI=alphaJ_in_proof_Case(alpha)}
   \prod_{j\in I} \alpha_j= \prod_{j\in J} \alpha_j.
\end{equation}
Using \eqref{equ:represent_alphai_final_step_in_Case(alpha)}, we represent the both sides of \eqref{equ:alphaI=alphaJ_in_proof_Case(alpha)} with $\eta_1,\dots,\eta_{s-1}.$ For $ i\in\{(s-1)/2,\dots,s-1\}\setminus\{\tau\},$ since the exponent of $\eta_i $ in the left hand side of \eqref{equ:alphaI=alphaJ_in_proof_Case(alpha)} is zero, we see $ J\cap M_i=\emptyset.$ Then by observing the exponents of $\eta_{\tau}$ in the both sides of \eqref{equ:alphaI=alphaJ_in_proof_Case(alpha)}, we see $ J\cap M_{\tau}$ has exactly one element. By $\# J=\# I $ and $ J\neq I,$ we get
\[
   J=M_0\cup\cdots\cup M_{(s-3)/2}\cup\{v_{\tau}\}.
\]
Cancelling the elements $\alpha_j\,(j\in M_0\cup\cdots\cup M_{(s-3)/2})$ in the both sides of \eqref{equ:alphaI=alphaJ_in_proof_Case(alpha)}, we get $\alpha_{u_{\tau}}=\alpha_{v_{\tau}}.$

So we have proved that
\[
   \alpha_{u_{\tau}}=\alpha_{v_{\tau}},\quad 1\leq \tau\leq s-1.
\]
Then as before, putting
\[
   \beta_{\tau}:=\alpha_{u_{\tau}},\quad 1\leq \tau\leq s-1,
\]
we see $\{\beta_1,\dots,\beta_{s-1}\}$ is a basis of $\langle\alpha_1,\dots,\alpha_{2s}\rangle,$ and after a suitable change of indices, we have
\[
   \alpha_1:\alpha_2:\dots:\alpha_{2s}= 1:1:\beta_1:\beta_1:\beta_2:\beta_2:\cdots:\beta_{s-1}:\beta_{s-1},
\]
which shows $(\alpha_1,\dots,\alpha_{2s})$ is of the type (A). This finishes the proof for the Case $(\alpha).$

\subsection{The proof for the Case $(\beta)$}        \label{ssec:Proof_for_the_Case(beta)}

Changing the indices of $\eta_1,\dots,\eta_t $ if necessary, we may assume that $ l(i,t)<0 $ for some $ i\in\{1,\dots,q\}.$

Put
\begin{align*}
   N_+ &:=\big\{i\in\{1,\dots,q\}\,|\, l(i,t)>0\big\},\quad n_+:=\# N_+,
\\ N_0 &:=\big\{i\in\{1,\dots,q\}\,|\, l(i,t)=0\big\},\quad n_0:=\# N_0,
\\ N_- &:=\big\{i\in\{1,\dots,q\}\,|\, l(i,t)<0\big\},\quad n_-:=\# N_-,
\end{align*}
where $ l(i,t)$ are the integers in \eqref{equ:represent_alphai_by_an_adequate_basis_after_tensor_Q}. Because $ l(u_t,t)=l_t>0,$ and $ M_0\subseteq N_0,$ and $ u_1,\dots, u_{t-1}\in N_0,$ we see that
\[
   n_+\geq 1,\quad n_-\geq 1, \quad\mbox{and}\quad n_0\geq m_0+(t-1)\geq t+1.
\]

Put $\tilde{n}:=\min\{n_+,n_-\}$ and assume $ N_0=\{p_1,\dots,p_{n_0}\}.$

\begin{claim}      \label{claim:alphai_in_N0_satisfy_induction_hypothesis}
$ 2\leq s-\tilde{n}<n_0\leq 2(s-\tilde{n}),$ and the $ n_0$-tuple $(\alpha_{p_1},\dots,\alpha_{p_{n_0}})$ has the property $(P_{n_0,s-\tilde{n}}).$
\end{claim}

\begin{proof}
Since $ q\leq 2s,$ we get
\[
   2s\geq q=n_0+n_++n_-\geq n_0+2\tilde{n},
\]
which implies $ n_0\leq 2(s-\tilde{n}).$ By $ t\geq s-1,$ we get
\[
   n_0\geq t+1\geq s.
\]
Thus $ n_0>s-\tilde{n}.$ Since $ s\geq 3,$ we know that
\[
   2s\geq n_0+n_++n_-\geq 3+2\tilde{n},
\]
which implies $ s-\tilde{n}\geq 2.$

Next we prove that $(\alpha_{p_1},\dots,\alpha_{p_{n_0}})$ has the property $(P_{n_0,s-\tilde{n}}).$

Consider first the case when $\tilde{n}=n_+.$ Take a subset $ I $ of $ N_0 $ with $\# I=s-n_+.$ Since $\# (I\cup N_+)=s,$ by assumption, there exists a subset $ J_1 $ of $\{1,\dots,q\}$ with $\# J_1=s $ and $ J_1\neq (I\cup N_+),$ such that
\begin{equation}      \label{equ:alpha(IcupN+)=alphaJ1}
   \prod_{i\in (I\cup N_+)}\alpha_i= \prod_{i\in J_1}\alpha_i.
\end{equation}
Regarding the $\alpha_i $ as elements in $\langle\alpha_1,\dots,\alpha_q \rangle\otimes_{\mathbb{Z}}\mathbb{Q}$ and substituting \eqref{equ:represent_alphai_by_an_adequate_basis_after_tensor_Q} into \eqref{equ:alpha(IcupN+)=alphaJ1}, we represent the both sides of \eqref{equ:alpha(IcupN+)=alphaJ1} with $\eta_1,\dots, \eta_t.$ Observing the exponents of $\eta_t,$ we see
\[
   J_1\supseteq N_+ \quad\mbox{and}\quad J_1\cap N_-=\emptyset.
\]
Then $ J_1\setminus N_+=:J\subseteq N_0 $ and $\# J=s-n_+.$ By $ J_1\neq (I\cup N_+),$ we also have $ J\neq I.$ Cancelling the elements $\alpha_i\, (i\in N_+)$ in the both sides of \eqref{equ:alpha(IcupN+)=alphaJ1}, we get
\[
   \prod_{i\in I}\alpha_i= \prod_{i\in J}\alpha_i.
\]
This shows $(\alpha_{p_1},\dots,\alpha_{p_{n_0}})$ has the property $(P_{n_0,s-n_+}).$

When $\tilde{n}=n_-,$ similar argument shows that $(\alpha_{p_1},\dots,\alpha_{p_{n_0}})$ has the property $(P_{n_0,s-n_-}).$ Therefore the Claim \ref{claim:alphai_in_N0_satisfy_induction_hypothesis} is proved.
\end{proof}

\begin{claim}        \label{claim:the_subgroup_generated_by_alphaN0_is_of_rank_t-1}
The subgroup $\langle\alpha_{p_1},\dots,\alpha_{p_{n_0}}\rangle $ of $ G $ generated by $\alpha_{p_1},\dots,\alpha_{p_{n_0}}$ is of rank $(t-1).$
\end{claim}

\begin{proof}
For any element $\beta\in\langle\alpha_{p_1},\dots,\alpha_{p_{n_0}}\rangle,$ if we regard $\beta $ as an element in $\langle\alpha_1,\dots,\alpha_q \rangle\otimes_{\mathbb{Z}}\mathbb{Q},$ then it can be represented by $\eta_1,\dots,\eta_{t-1}.$ This fact implies that any $ t $ elements in $\langle\alpha_{p_1},\dots,\alpha_{p_{n_0}}\rangle $ are multiplicatively dependent, and therefore
\[
   {\rm rank}\{\alpha_{p_1},\dots,\alpha_{p_{n_0}}\}\leq t-1.
\]

On the other hand, we know $\alpha_{u_1},\dots,\alpha_{u_{t-1}}\in\langle\alpha_{p_1},\dots,\alpha_{p_{n_0}}\rangle.$ And, as elements in $\langle\alpha_1,\dots,\alpha_q \rangle\otimes_{\mathbb{Z}}\mathbb{Q},$
\[
   \alpha_{u_1}=\eta_1^{l_1},\dots,\alpha_{u_{t-1}}=\eta_{t-1}^{l_{t-1}}.
\]
Thus, $\alpha_{u_1},\dots,\alpha_{u_{t-1}}$ are multiplicatively independent elements in $ G $ and it follows that
\[
   {\rm rank}\{\alpha_{p_1},\dots,\alpha_{p_{n_0}}\}\geq t-1.
\]
So we get ${\rm rank}\{\alpha_{p_1},\dots,\alpha_{p_{n_0}}\}=t-1,$ which proves Claim \ref{claim:the_subgroup_generated_by_alphaN0_is_of_rank_t-1}.
\end{proof}

By Claims \ref{claim:alphai_in_N0_satisfy_induction_hypothesis} and \ref{claim:the_subgroup_generated_by_alphaN0_is_of_rank_t-1}, we can apply the induction hypothesis to the $ n_0$-tuple $(\alpha_{p_1},\dots,\alpha_{p_{n_0}})$ and conclude that
\[
   t-1={\rm rank}\{\alpha_{p_1},\dots,\alpha_{p_{n_0}}\}\leq s-\tilde{n}-1\leq s-2.
\]
Then by $ t\geq s-1,$ we see
\[
   t=s-1,\quad \tilde{n}=1, \quad\mbox{and}\quad {\rm rank}\{\alpha_{p_1},\dots,\alpha_{p_{n_0}}\}=s-2.
\]
Applying the induction hypothesis to the $ n_0$-tuple $(\alpha_{p_1},\dots,\alpha_{p_{n_0}})$ again, we see that $ n_0=2(s-1)$ and $(\alpha_{p_1},\dots,\alpha_{p_{n_0}})$ is of the type (A) or of the type (B). So
\[
   2s\geq q=n_0+n_++n_-\geq 2(s-1)+2=2s.
\]
It follows that $ q=2s $ and $ n_+=n_-=1.$

We have proved that the conclusion (i) of Lemma \ref{lem:the_third_combinatorial_lemma} holds.
We shall show next that $(\alpha_1,\dots,\alpha_{2s})$ is of the type (B), and this implies that the conclusion (ii) of Lemma \ref{lem:the_third_combinatorial_lemma} also holds.

First, we show the following.
\begin{claim}       \label{claim:alphaN0_is_not_of_type(A)}
It is impossible that the $(2s-2)$-tuple $(\alpha_{p_1},\dots,\alpha_{p_{n_0}})$ is of the type {\rm (A)}.
\end{claim}

\begin{proof}
We assume the $(2s-2)$-tuple $(\alpha_{p_1},\dots,\alpha_{p_{n_0}})$ is of the type (A).
Then $(s-1)$ is odd and thus $ s $ is an even integer with $ s\geq 4.$

Since $\{1,\alpha_{u_1},\dots,\alpha_{u_{s-2}}\}\subseteq\{\alpha_{p_1},\dots,\alpha_{p_{n_0}}\},$ when we regard the $\alpha_i $ as elements in $\langle\alpha_1,\dots,\alpha_q \rangle\otimes_{\mathbb{Z}}\mathbb{Q},$ we may assume, after a suitable change of indices, that
\begin{equation}     \label{equ:alphai_when_alphaN0_is_of_type(A)}
(\alpha_1,\dots,\alpha_{2s})=(1,1,\eta_1^{l_1},\eta_1^{l_1},\dots,\eta_{s-2}^{l_{s-2}},\eta_{s-2}^{l_{s-2}},
             \eta_{s-1}^{l_{s-1}},\eta_1^{l'_1}\cdots\eta_{s-1}^{l'_{s-1}}),
\end{equation}
where $ l_1,\dots, l_{s-1}$ are positive integers in \eqref{equ:represent_alphai_by_an_adequate_basis_after_tensor_Q}, and $ l'_1,\dots, l'_{s-1}$ are integers and $ l'_{s-1}<0.$

For convenience' sake, we write
\[
   B_0:=\{1,2\},B_1:=\{3,4\},\dots,B_{s-2}:=\{2s-3,2s-2\},B_{s-1}:=\{2s-1,2s\}.
\]

We note that $ s/2 $ is a positive integer and $ s/2\leq s-2.$ Put
\[
   I:=B_1\cup\cdots\cup B_{s/2}.
\]
Since $\# I=s,$ by assumption, there exists a $ J\subseteq\{1,\dots,2s\}$ with $\# J=s $ and $ J\neq I,$ such that
\begin{equation*}
   \prod_{j\in I} \alpha_j= \prod_{j\in J} \alpha_j.
\end{equation*}
Substituting \eqref{equ:alphai_when_alphaN0_is_of_type(A)} into the above equation, we represent the both sides of the above equation with $\eta_1,\dots,\eta_{s-1}.$
If $ J\cap B_{s-1}=\emptyset,$ by observing the exponents of $\eta_i\, (1\leq i\leq s/2),$ we see $ J\supseteq B_1\cup\cdots\cup B_{s/2}$ and thus $ J=I,$ which is a contradiction. Therefore $ J\cap B_{s-1}\neq\emptyset.$ Then by observing the exponents of $\eta_{s-1},$ we see $ J\supseteq B_{s-1}$ and $ l_{s-1}+l'_{s-1}=0.$ By $\# J=\# I,$ there exists an index $ i_0 $ with $ 1\leq i_0\leq s/2 $ such that $ J\not\supseteq B_{i_0}.$ Then by observing the exponents of $\eta_{i_0},$ it is necessary that $ l'_{i_0}=l_{i_0}$ or $ l'_{i_0}=2l_{i_0}.$ In particular, we get $ l'_{i_0}>0.$

Now we put
\[
   I':=B_0\cup B_1\cup\cdots\cup B_{i_0-1}\cup B_{i_0+1}\cup\cdots\cup B_{s/2}.
\]
By assumption, there exists a $ J'\subseteq\{1,\dots,2s\}$ with $\# J'=s $ and $ J'\neq I',$ such that
\begin{equation}      \label{equ:alphaI=alphaJ_in_proof_alphaN0_is_not_of_type(A)}
   \prod_{j\in I'} \alpha_j= \prod_{j\in J'} \alpha_j.
\end{equation}
Substituting \eqref{equ:alphai_when_alphaN0_is_of_type(A)} into \eqref{equ:alphaI=alphaJ_in_proof_alphaN0_is_not_of_type(A)}, we represent the both sides of \eqref{equ:alphaI=alphaJ_in_proof_alphaN0_is_not_of_type(A)} with $\eta_1,\dots,\eta_{s-1}.$
If $ J'\cap B_{s-1}=\emptyset,$ by observing the exponents of $\eta_i $ with $ i\in\{i_0,s/2+1,\dots,s-2\},$ we see
\[
   J'\cap (B_{i_0}\cup B_{s/2+1}\cup\cdots\cup B_{s-2})=\emptyset
\]
and thus $ J'=I',$ which is a contradiction. Thus $ J'\cap B_{s-1}\neq\emptyset,$ and as before we get $ J'\supseteq B_{s-1}.$
Then the exponent of $\eta_{i_0}$ in the right hand side of \eqref{equ:alphaI=alphaJ_in_proof_alphaN0_is_not_of_type(A)} is strictly greater than zero. However the exponent of $\eta_{i_0}$ in the left hand side of \eqref{equ:alphaI=alphaJ_in_proof_alphaN0_is_not_of_type(A)} is zero. We get a contradiction and thus the conclusion of Claim \ref{claim:alphaN0_is_not_of_type(A)} holds.
\end{proof}

Now we know that the $(2s-2)$-tuple $(\alpha_{p_1},\dots,\alpha_{p_{n_0}})$ is of the type (B).

Put
\[
   \beta_1:=\alpha_{u_1},\,\beta_2:=\alpha_{u_2},\dots, \beta_{s-1}:=\alpha_{u_{s-1}}.
\]
Then $\beta_1,\dots,\beta_{s-1}$ are multiplicatively independent elements in $ \langle\alpha_1,\dots,\alpha_{2s}\rangle,$ and as elements in $\langle\alpha_1,\dots,\alpha_{2s}\rangle\otimes_{\mathbb{Z}}\mathbb{Q},$ they can be represented as follows:
\[
   \beta_1=\eta_1^{l_1},\dots,\beta_{s-2}=\eta_{s-2}^{l_{s-2}},\,\beta_{s-1}=\eta_{s-1}^{l_{s-1}},
\]
where $ l_1,\dots, l_{s-1}$ are positive integers.

Since $\{\beta_1,\dots,\beta_{s-2}\}\subseteq\{\alpha_{p_1},\dots,\alpha_{p_{n_0}}\},$ by Observation \ref{obsv:type(B)_is_symmetry_for_nonunit_elements} and suitably changing the indices of $\eta_1,\dots,\eta_{s-2}$ and of the $\alpha_i $'s, we may assume that
\begin{align*}
   (\alpha_1,\dots,\alpha_{2s})&
\\  =\big(1,1,&\dots,\underset{\underset{\alpha_{s-k}}{\uparrow}}{1},\underset{\underset{\alpha_{s-k+1}}{\uparrow}}{\beta_1},\dots, \underset{\underset{\alpha_{2s-k-2}}{\uparrow}}{\beta_{s-2}}, \underset{\underset{\alpha_{2s-k-1}}{\uparrow}}{(\beta_1\cdots\beta_{a_1})^{-1}},
\\ &\qquad (\beta_{a_1+1}\cdots\beta_{a_2})^{-1},\dots,\underset{\underset{\alpha_{2s-2}}{\uparrow}}{(\beta_{a_{k-1}+1}\cdots\beta_{a_k})^{-1}}, \underset{\underset{\alpha_{2s-1}}{\uparrow}}{\eta_{s-1}^{l_{s-1}}}, \underset{\underset{\alpha_{2s}}{\uparrow}}{\eta_1^{l'_1}\cdots\eta_{s-1}^{l'_{s-1}}}\big),
\end{align*}
where $ 0\leq k\leq s-2 $, $ 1\leq a_1<\dots<a_k\leq s-2 $, and $ l'_1,\dots, l'_{s-1}$ are integers and $ l'_{s-1}<0.$

We shall prove next the following assertions which imply that $(\alpha_1,\dots,\alpha_{2s})$ is of the type (B):
\begin{equation}      \label{equ:l'i_in_proof_alphaj_is_of_type(B)}
 \begin{split}
    & l'_{s-1}=-l_{s-1},
 \\ & l'_i\in\{0,-l_i\} \quad\mbox{for}\quad a_k+1\leq i\leq s-2,
 \\ & l'_i=0 \quad\mbox{for}\quad 1\leq i\leq a_k.
 \end{split}
\end{equation}

Put
\[
   I:=\{1,2,\dots,s-k,\, 2s-k-1,\dots,2s-2\}.
\]
Since $\# I=s,$ by assumption, there exists a $ J\subseteq\{1,\dots,2s\}$ with $\# J=s $ and $ J\neq I,$ such that
\begin{equation*}
   \prod_{j\in I} \alpha_j= \prod_{j\in J} \alpha_j.
\end{equation*}
We represent the both sides of the above equation with $\eta_1,\dots,\eta_{s-1}.$ If $ J\cap\{2s-1,2s\}=\emptyset,$ by observing the exponents of $\eta_i\, (1\leq i\leq s-2),$ we see $ J\cap\{s-k+1,\dots,2s-k-2\}=\emptyset,$ and thus $ J=I,$ which is a contradiction. Therefore $ J\cap\{2s-1,2s\}\neq\emptyset.$ Then by observing the exponents of $\eta_{s-1},$ we see $ J\supseteq\{2s-1,2s\}$ and $ l_{s-1}+l'_{s-1}=0.$ Assume $ l'_i\neq 0 $ for some $ a_k+1\leq i\leq s-2.$ Since the exponent of $\eta_i $ in the left hand side is zero, we see that $ s-k+i\in J $ and $ l_i+l'_i=0.$ Assume that $ l'_i>0 $ for some $ 1\leq i\leq a_k,$ then we see that the exponent of $\eta_i $ in the right hand side is strictly greater than that in the left hand side and this is a contradiction.

In summary, we have proved that $ l'_{s-1}=-l_{s-1},$ and $ l'_i\in\{0,-l_i\}$ for $ a_k+1\leq i\leq s-2,$ and $ l'_i\leq 0 $ for $ 1\leq i\leq a_k.$

We shall show next that, for any $ 1\leq i\leq a_k,$ it is impossible that $ l'_i<0.$ Without loss of generality, we assume that $ l'_1<0.$ Put
\[
   I':=\{1,2,\dots,s-k,\, s-k+1,\, 2s-k,\dots,2s-2\}.
\]
By assumption, there exists a $ J'\subseteq\{1,\dots,2s\}$ with $\# J'=s $ and $ J'\neq I',$ such that
\begin{equation*}
   \prod_{j\in I'} \alpha_j= \prod_{j\in J'} \alpha_j.
\end{equation*}
We represent the both sides of the above equation with $\eta_1,\dots,\eta_{s-1}.$ Assume first that $ J'\cap\{2s-1,2s\}=\emptyset.$ By observing the exponents of $\eta_i $ with $ i\in\{1,a_1+1,a_1+2,\dots,s-2\},$ we see
\[
   J'\cap\{s-k+a_1+1,\, s-k+a_1+2,\dots,2s-k-1\}=\emptyset.
\]
Then since $ 2s-k-1\not\in J',$ by observing the exponents of $\eta_i\, (2\leq i\leq a_1),$ we see $ J'\cap\{s-k+2,\dots,s-k+a_1\}=\emptyset.$ Thus by $\# J'=\# I',$ we get $ J'=I',$ which is a contradiction. Therefore $ J'\cap\{2s-1,2s\}\neq\emptyset,$ and as before, we get $ J'\supseteq\{2s-1,2s\}.$ Since $ l'_1<0,$ we see that the exponent of $\eta_1 $ in the right hand side is strictly less than that in the left hand side. We get a contradiction and thus \eqref{equ:l'i_in_proof_alphaj_is_of_type(B)} is proved.

By \eqref{equ:l'i_in_proof_alphaj_is_of_type(B)}, we conclude that there is a subset $\Lambda $ of $\{a_k+1,\dots,s-1\}$ with $ s-1\in\Lambda $ such that
\[
   \alpha_{2s}=(\prod_{i\in\Lambda} \beta_i)^{-1}.
\]
Thus, suitably changing indices, we have
\begin{align*}
   \alpha_1:\alpha_2:\dots:\alpha_{2s}&=1:1:\dots:1:\beta_1:\dots:\beta_{s-1}:
\\ &(\beta_1\cdots\beta_{a_1})^{-1}:\dots:(\beta_{a_{k-1}+1}\cdots\beta_{a_k})^{-1}:(\beta_{a_k+1}\cdots\beta_{a_{k+1}})^{-1},
\end{align*}
where $ 1\leq k+1\leq s-1 $, $ 1\leq a_1<\dots<a_k<a_{k+1}\leq s-1,$ and $\{\beta_1,\dots,\beta_{s-1}\}$ is a basis of $\langle\alpha_1,\dots,\alpha_{2s}\rangle.$ This shows that $(\alpha_1,\dots,\alpha_{2s})$ is of the type (B) and therefore we complete the proof of Lemma \ref{lem:the_third_combinatorial_lemma}.

% -------------------------------------------------------------------------
\bibliographystyle{plain}
\bibliography{zkRef}

\begin{thebibliography}{1}

\bibitem{Fujimoto75}
Hirotaka Fujimoto.
\newblock The uniqueness problem of meromorphic maps into the complex
  projective space.
\newblock {\em Nagoya Math. J.}, 58:1--23, 1975.

\bibitem{Fujimoto76}
Hirotaka Fujimoto.
\newblock A uniqueness theorem of algebraically non-degenerate meromorphic maps
  into {$P^N(\mathbf{C})$}.
\newblock {\em Nagoya Math. J.}, 64:117--147, 1976.

\bibitem{Fujimoto78}
Hirotaka Fujimoto.
\newblock Remarks to the uniqueness problem of meromorphic maps into
  {$P^N(\mathbf{C})$, I}.
\newblock {\em Nagoya Math. J.}, 71:13--24, 1978.

\bibitem{Nevanlinna1926}
Rolf Nevanlinna.
\newblock {Einige Eindeutigkeitss\"atze in der Theorie der meromorphen
  Funktionen}.
\newblock {\em Acta Math.}, 48(3-4):367--391, 1926.

\bibitem{Polya1921}
George P\'olya.
\newblock {Bestimmung einer ganzen Funktion endlichen Geschlechts durch
  viererlei Stellen}.
\newblock {\em Mat. Tidsskrift B}, pages 16--21, 1921.
\newblock https://zbmath.org/?q=an\%3A48.0354.03.

\end{thebibliography}

\end{document}